\documentclass{article}

\usepackage{color}
\usepackage{xcolor}
\usepackage{amsthm,amssymb,amsfonts,mathrsfs}
\usepackage{amsmath}

\newcommand{\RN}[1]{%
	\textup{\uppercase\expandafter{\romannumeral#1}}%
}

\newtheorem{Theorem}{Theorem}[section]
\newtheorem{Proposition}{Proposition}[section]
\newtheorem{Lemma}{Lemma}[section]

\theoremstyle{definition}

\newtheorem{Remark}{Remark}

\newtheorem{Assumptions}{Hypothesis}[section]

\def\N{{\mathbb{N}}}

\def\R{{\mathbb{R}}}

\def\cU{{\mathcal{U}}}

\def\t\cU{{\widetilde{{\mathcal{U}}}}}

\newcommand\norm[1]{\left\lVert#1\right\rVert}

\DeclareMathOperator{\supp}{supp}

\def\ds{\displaystyle}

\title {A degenerate operator in non divergence form}
\author{{{\sc Alessandro Camasta}\thanks{The author is a member of the  {\it Gruppo Nazionale per l'Analisi Ma\-te\-matica, la Probabilit\`a e le loro Applicazioni (GNAMPA)} of the Istituto Nazionale di Alta Matematica (INdAM) and a member of UMI  ``Modellistica Socio-Epidemiologica (MSE)''. He is partially supported by PRIN 2017-2019 {\it Qualitative and quantitative aspects of nonlinear PDEs.}}}\\
Department of Mathematics\\ University of Bari Aldo Moro\\
Via
E. Orabona 4\\ 70125 Bari - Italy\\ e-mail: alessandro.camasta@uniba.it\\
{\sc Genni Fragnelli}\thanks{The author is a member of the  {\it Gruppo Nazionale per l'Analisi Ma\-te\-matica, la Probabilit\`a e le loro Applicazioni (GNAMPA)} of the Istituto Nazionale di Alta Matematica (INdAM) and a member of {\it UMI ``Modellistica Socio-Epidemiologica (MSE)''}. She is partially supported by the FFABR {\it Fondo per il finanziamento delle attivit\`a base di ricerca} 2017 and by PRIN 2017-2019 {\it Qualitative and quantitative aspects of nonlinear PDEs.}}\\
Department of Ecology and Biology\\ Tuscia University\\
Largo dell'Universit\`a, 01100 Viterbo - Italy\\ e-mail: genni.fragnelli@unitus.it}

\date{}
\begin{document}

\maketitle
\centerline{\emph{In honour of Francesco Altomare,}}
\centerline{\emph{with deep affection on his 70th birthday}}

\vspace{1cm}

\begin{abstract}
	In this paper we consider a fourth order operator in non divergence form $Au:= au''''$, where $a: [0,1] \rightarrow \R_+$ is a function that degenerates somewhere in the interval. We prove that the operator generates an analytic semigroup, under suitable assumptions on the function $a$. We extend these results to a general operator  $A_nu := au^{(2n)}$.
\end{abstract}


\noindent 2000AMS Subject Classification: Primary: 47D06, 35K65; Secondary: 47B25, 47N20.

\section{Introduction}
In this paper we analyze the properties of a degenerate fourth order  differential operator in non divergence form under Dirichlet boundary conditions in the real setting. More precisely, we consider the operator $Au:= au''''$ with a suitable domain, where we denote with $'$ the derivative of a function depending only on one variable $x$, which we assume to vary in $[0,1]$. The coefficient $a$ is a function for which the degeneracy may occur in the interior of the interval or on the boundary of it. 

In fact, we shall admit two types of degeneracy for $a$, namely weak and strong degeneracy. In particular, following \cite{acf}, \cite{3} or \cite{fmAnona2013}, we say that a function $g:[0,1]\to\mathbb{R}$ is 
\begin{itemize}
	\item {\it weakly degenerate at $x_0 \in [0, 1]$} if $g \in \mathcal{C}[0, 1]$, $g(x_0) = 0$, $g > 0$ on $[0, 1]\setminus\{x_0\}$ and $\frac{1}{g}\in L^1(0,1)$;
	\item {\it strongly degenerate at $x_0 \in [0, 1]$} if $g \in \mathcal{C}[0, 1]$, $g(x_0) = 0$, $g > 0$ on $[0, 1]\setminus \{x_0\}$ and $\frac{1}{g}\not\in L^1(0,1)$.
\end{itemize}

We are interested in this type of operators since many problems that are relevant for applications are described by fourth order operators. Among these applications we can find dealloying (corrosion processes, see, e.g., \cite{eakds}), population dynamics (see, e.g., \cite{cm}), bacterial films (see, e.g., \cite{kd}), thin film  (see, e.g., \cite{odb}), Chemistry  (see, e.g., \cite{vbd}), tumor growth  (see, e.g., \cite{akw}, \cite{ks}), image processing (denoising, inpainting, see, e.g., \cite{be}, \cite{c}, \cite{dv}), Astronomy (rings of Saturn, see, e.g., \cite{t}), Ecology (surprisingly, the clustering of mussels can be perfectly well described by the Cahn–Hilliard equation, see, e.g, \cite{l}) and so on.

Let us present very  briefly some interesting results about the existence and uniqueness of solutions for problems associated to the operators under consideration. In \cite{ksw} the authors  study the  epitaxial growth of nanoscale thin films which can be described by a parabolic equation of the form
\[
\frac{\partial u}{\partial t} +\Delta^2 u - \nabla\cdot(f(\nabla u)) =g
\]
in $(0,T)\times (0,L)$, where $f$ and $g$ belong to $\mathcal{C}^1(\R^N, \R^N)$, $N\ge2$, and $L^2((0,T)\times (0,L))$, respectively. The authors show existence, uniqueness and regularity of solutions in suitable functional spaces. In \cite{dgg} the authors consider a degenerate fourth order operator of the form $\nabla \cdot(m(u)\nabla \Delta u)$ where $m$ is a specific function, proving existence and non uniqueness results for the parabolic equation associated to this operator (see also \cite{bbdal}, \cite{bf}, \cite{bp}, \cite{eg}, \cite{gr} or \cite{ll}). In \cite{gg} the existence of a weak  solution for the following equation is proved
\[
\frac{\partial u}{\partial t}+ \nabla \cdot(|\nabla \Delta u|^{p(x)-2} \nabla \Delta u)=f(x,u),
\]
in $(0,T)\times\Omega$, under the conditions
$u=\Delta u=0$ on $\partial \Omega$ and 
$u(0,x)=u_0(x)$, $x\in \Omega \subset \R^N$, $N\ge 2$. Here $p$ and $f$ are specific functions and $u_0$ is an initial datum. Observe that for $p\equiv2$ the parabolic problem associated to the previous operator becomes the classical Cahn–Hilliard problem, which has been extensively studied (see, e.g., \cite{lz}). The previous model can describe some properties of medical magnetic resonance images in space and time. In particular, if $f(x,u):= u(t,x)-a(x)$, then $u$ represents a digital image and $a$ its observation.

Recently, in \cite{fgmr} the general operator $\tilde A_nu:= (au^{(n)})^{(n)}$ is considered, where $a\in \mathcal{C}[0,1]$ degenerates  in an interior point $x_0$. The authors give sufficient conditions on the function $a$ so that the operator $(\tilde A_n, D(\tilde A_n))$ generates a contractive analytic semigroup on $L^2(0,1)$.

As for the second order operator (see, e.g., \cite{cfr}, \cite{2} or \cite{4} and the references therein), the generation property for the operator in non divergence form cannot be deduced by the one of the operator in divergence form without assuming other assumptions on the function $a$. Moreover, another difference between the operator in divergence form and the one in non divergence form is the fact that the natural space to study the equation associated to  $\tilde Au:= (au'')''$, or $\tilde A_nu:= (au^{(n)})^{(n)}$,  is $L^2(0,1)$ whereas the problem associated to $Au:= au''''$, or more in general to $A_nu:= au^{(2n)}$, is more conveniently set in the weighted space
\[
L^2_{\frac{1}{a}}(0,1):=\left \{ u\in L^2(0,1): \int_0^1\frac{u^2(x)}{a(x)}\,dx \in \R\right\}.
\]
In such a space we will prove that the fourth order operator $A$ and, in general, the operator $A_n$ generate an analytic semigroup.

The paper is organized in the following way. In Section \ref{SECTION 2} we assume that the degeneracy point belongs to the boundary of the space domain and  we consider the fourth order operator, proving some preliminary results that will be crucial to prove the generation property of it in Theorem \ref{Theorem 2.1}. In Section \ref{Section 3} we characterize the domain of the operator in the weakly and in the strongly degenerate case under additional assumptions on the degenerate function $a$. Thanks to the characterization of the domain we prove again the generation property, if the degeneracy point is in the interior of the domain.  In Section \ref{Section4} we extend the previous results to the general operator $A_nu= au^{(2n)}$, $n \ge 3$.

A final comment on the notation: by  $C$ we shall denote universal positive
constants, which are allowed to vary from line to line. 

This paper is a tribute to Professor Francesco Altomare for celebrating his 70th birthday and for thanking him for the wonderful teaching and research activities realized with great efficiency, accuracy and passion.
\section{The fourth order operator if the degeneracy point belongs to the boundary}\label{SECTION 2}
In this section we introduce the operator $Au:=au''''$, where $a:[0,1]\to\mathbb{R}_+$ is a given function that degenerates somewhere in the space domain, and we consider the following (weighted) Hilbert spaces:
\begin{equation*}
	L^2_{\frac{1}{a}}(0, 1):=\biggl \{u\in L^2(0, 1):\int_{0}^{1}\frac{u^2}{a}\,dx<+\infty \biggr \}
\end{equation*}
and
\begin{equation*}
	H^i_{\frac{1}{a}}(0, 1):= L^2_{\frac{1}{a}}(0, 1)\cap H^i_0(0, 1), \quad i=1,2,
\end{equation*}
with the norms
\begin{equation*}
	\norm{u}^2_{L^2_{\frac{1}{a}}(0, 1)}:= \int_{0}^{1}\frac{u^2}{a}\,dx\quad\,\,\,\,\,\,\, \forall \; u \in L^2_{\frac{1}{a}}(0, 1)
\end{equation*}
and
\begin{equation*}
	\norm{u}_{H^i_{\frac{1}{a}}(0, 1)}^2:=\norm{u}^2_{L^2_{\frac{1}{a}}(0, 1)} + \sum_{k=1}^{i}\Vert u^{(k)}\Vert^2_{L^2(0, 1)}\quad\,\,\,\,\,\,\, \forall \; u \in H^i_{\frac{1}{a}}(0, 1),
\end{equation*}
$i=1,2$, respectively.
We recall that $H^i_0(0, 1):= \{u\in H^i(0, 1): u^{(k)}(j)=0,\, j=0,1,\, k=0,...,i-1\}$, with $u^{(0)}=u$ and $i=1,2$.

Observe that for all $u \in H^i_{\frac{1}{a}}(0, 1)$, using the fact that $u^{(k)}(j)=0$ for all $k=0,...,i-1$ and $j=0,1$, it is easy to prove that $\|u\|_{H^i_{\frac{1}{a}}(0, 1)}^2$ is equivalent to the following one 
\[
\| u \|_i ^2:= \|u\|_{L^2_{\frac{1}{a}}(0, 1)}^2+ \|u^{(i)}\|_{L^2(0,1)}^2.
\]
Thus, for simplicity, in the rest of the paper we will use $\|\cdot\|_i$ in place of $\|\cdot\|_{H^i_{\frac{1}{a}}(0, 1)}$.

Using the previous spaces, it is possible to define the operator $A$ by 
\begin{equation}\label{operator}
	Au:=au'''' \quad \text{for all } u \in
	D(A):=\left\{u\in H^2_{\frac{1}{a}}(0, 1): au''''\in L^2_{\frac{1}{a}}(0, 1) \right\},
\end{equation}
if $x_0 \in \{0,1\}$. The case $x_0 \in (0,1)$ will be considered in the next section.
In order to prove that $(A,D(A))$ generates a semigroup, we assume that  $a$ satisfies the following hypothesis:
\begin{Assumptions}\label{ipo1}
	The function $a$ belongs to the space of continuous functions $\mathcal C[0,1]$  and there exists a point $x_0\in \{0,1\}$ such that  $a(x_0)=0$ and $a>0$ on
	$[0,1]\setminus \{x_0\}$.
\end{Assumptions}

\begin{Proposition}[Green's Formula]\label{Proposition Green's Formula}
	
	Assume Hypothesis \ref{ipo1}. For all $(u,v)\in D(A)\times H^2_{\frac{1}{a}}(0, 1)$ one has
	\begin{equation}\label{GF1}
		\int_{0}^{1}u''''v\,dx=\int_{0}^{1}u''v''dx.
	\end{equation}
\end{Proposition}
The proof of the previous proposition is based on the next result which is standard, but here we give it for the reader's convenience.

\begin{Lemma}\label{derivatadistribuzionale} Let $I:= (\alpha, \beta)$, with $\alpha, \beta\in\mathbb{R}$, $\alpha < \beta$, $p\ge 1$ and $\mathscr D'(I)$ the space of distributions defined on $I$. If $f\in \mathscr D'(I)$ has n-th derivative which is a function belonging to $L^p(I)$, then $f \in W^{n,p}(I)$, where $n\in\mathbb{N}$, $n\ge 1$.
\end{Lemma}
\begin{proof}
	Let $x \in I$  and define 
	\[
	\begin{cases}v_1(x):= \int_\alpha^x f^{(n)}(s)ds ,\\
		v_i(x):= \int_\alpha^x v_{i-1}(s)ds,
	\end{cases}
	\]
	$i=2, ...n$. Then, for all $i =1,...n$,
	\begin{equation}\label{induzione}
		v_i \in W^{i,p}(I)\,\,\, \text{ and }\,\,\, v_i^{(i)}= f^{(n)}.
	\end{equation} 
	Indeed, for $i=1$ the thesis is obvious.
	Now, we prove it for $i=2$:
	\[
	v_{2}''= (v_2')'= (v_1)'=  f^{(n)}.
	\]
	Hence, iterating the procedure, one has that \eqref{induzione} holds.

	Now, define $\psi:= v_n-f$. Hence the distributional $n$-th derivative of $\psi$ is given by $\psi^{(n)}= v_n^{(n)}-f^{(n)}=0$. Thus, there exists a constant $c\in \R$ such that
	\[
	f^{(n-1)}= v_n^{(n-1)}+c
	\]
	a.e. in $I$; this implies that $f^{(n-1)} \in W^{1,p}(I)$. In particular, one has $f^{(n-1)} \in L^{p}(I)$. Proceeding as in the first part of the proof and iterating the procedure, one has that $f^{(i)} \in  L^{p}(I)$ for all $i=1,2,...,n-2$. 
	
	Now, define $z(x):= \int_\alpha^xf'(s)ds$ and $w(x):= z(x)-f(x)$. Clearly, $z \in W^{2,p}(I)$  and $w'(x)=0$ a.e. in $I$. This implies that there exists a constant $C\in \R$ such that
	\[
	f= z+C.
	\] 
	In particular, $f \in L^p(I)$ and the thesis follows.
\end{proof}

\begin{proof}[Proof of Proposition \ref{Proposition Green's Formula}] 
	Following the idea of \cite[Lemma 2.1]{2}, one can prove that
	the space $H^2_c(0, 1):=\{v\in H^2(0, 1): \supp v\subset (0, 1) \}$ is dense in $H^2_{\frac{1}{a}}(0, 1)$.
	
	Indeed,  we can
	consider the sequence $(v_n)_{n \ge 4}$, where $v_n: =\xi_nv$ for a
	fixed function $v\in H^2_{{\frac{1}{a}}}(0,1)$ and
	\[
	\xi_n(x):=\left\{
	\begin{array}{ll}
		0,&x\in\; \left[0,1/n\right]\cup \left[1-1/n,1\right],
		\\
		1,&x\in \;  \left[2/n, 1 -2/n \right],
		\\
		-2n^3x^3+9n^2x^2-12nx +5, &x\in\;  \left( 1/n ,2/n \right),
		\\
		f(n,x),  &x\in\;   \left(1-2/n ,1-1/n \right).
	\end{array}
	\right.
	\]
	Here $f(n,x):= a_nx^3 +b_nx^2+c_nx+d_n$, being
	\[
	a_n:= \frac{2}{-\left(1-\frac{2}{n}\right)^3+3\left(1-\frac{1}{n}\right)\left(1-\frac{2}{n}\right)^2-3\left(1-\frac{1}{n}\right)^2\left(1-\frac{2}{n}\right)+ \left(1-\frac{1}{n}\right)^3},
	\]
	\[
	b_n:= \frac{-3\left(1-\frac{2}{n}\right)-3\left(1-\frac{1}{n}\right)}{-\left(1-\frac{2}{n}\right)^3+3\left(1-\frac{1}{n}\right)\left(1-\frac{2}{n}\right)^2-3\left(1-\frac{1}{n}\right)^2\left(1-\frac{2}{n}\right)+ \left(1-\frac{1}{n}\right)^3},
	\]
	\[
	c_n:= \frac{6\left(1-\frac{1}{n}\right)\left(1-\frac{2}{n}\right)}{-\left(1-\frac{2}{n}\right)^3+3\left(1-\frac{1}{n}\right)\left(1-\frac{2}{n}\right)^2-3\left(1-\frac{1}{n}\right)^2\left(1-\frac{2}{n}\right)+ \left(1-\frac{1}{n}\right)^3}
	\]
	and
	\[
	d_n:=\frac{\left(1-\frac{1}{n}\right)^3-3\left(1-\frac{1}{n}\right)^2\left(1-\frac{2}{n}\right)}{-\left(1-\frac{2}{n}\right)^3+3\left(1-\frac{1}{n}\right)\left(1-\frac{2}{n}\right)^2-3\left(1-\frac{1}{n}\right)^2\left(1-\frac{2}{n}\right)+ \left(1-\frac{1}{n}\right)^3}.
	\]
	It is easy to see that $v_n\rightarrow v$ in
	$L^2_{{\frac{1}{a}}}(0,1)$. Indeed, setting $g_n:= v_n-v$, one has that $\lim_{n \rightarrow +\infty}g_n = 0$ a.e. and $|g_n| \le 2|v| \in L^2_{\frac{1}{a}}(0,1)$ for all $n \in \mathbb{N}$, $n\ge 4$. Hence, by the Lebesgue Theorem, one can conclude that $v_n\rightarrow v$ in
	$L^2_{{\frac{1}{a}}}(0,1)$.
	Moreover, one has that
	\begin{equation}\label{h1c}
		\begin{aligned}
			\int_0^1((v_n- v)'')^2dx\
			&\le
			2\int_0^1(1-\xi_n)^2(v'')^2dx
			+2\int_{\frac{1}{n}}^{\frac{2}{n}}(\xi_n''v)^2dx
			\\
			&+2\int_{1-\frac{2}{n}}^{1-\frac{1}{n}}(\xi_n''v)^2dx
			+8\biggl (\int_{\frac{1}{n}}^{\frac{2}{n}}(\xi_n'v')^2dx +\int_{1-\frac{2}{n}}^{1-\frac{1}{n}}(\xi_n'v')^2dx \biggr ).
		\end{aligned}
	\end{equation}
	Obviously, proceeding as before, the first term in the last member of \eqref{h1c}
	converges to zero.
	Furthermore, since $v,v'\in H^1_0(0,1)$, by H\"older's inequality one has that
	$$
	v^2(x)\le x\int_0^x (v')^2(y)dy\qquad\forall\; x\in [0,1]
	$$
	and 
	$$
	(v')^2(y)\le y\int_0^y (v'')^2(z)dz\qquad\forall\; y\in [0,1]\,.
	$$
	Hence,
	\[
	\begin{aligned}
		v^2(x)&\le x\int_0^x\left( y\int_0^y (v'')^2(z)dz\right) dy\le x^2\int_0^x\int_0^x (v'')^2(z)dz\, dy\\
		&\le x^3\int_0^x (v'')^2(z)dz.
	\end{aligned}
	\]
	Using this inequality, one can prove that there exists a positive constant $C$ such that
	\begin{equation}\nonumber
		\begin{split}
			\int_{\frac{1}{n}}^{\frac{2}{n}}(\xi_n''v)^2dx
			&\le C \int_{\frac{1}{n}}^{\frac{2}{n}}(n^6x^2+ n^4)v^2(x)dx
			\\
			&\le C \int_{\frac{1}{n}}^{\frac{2}{n}}(n^6x^2+ n^4)x^3\left( \int_0^x(v'')^2(z)dz\right)dx
			\\
			&=
			C\int_{\frac{1}{n}}^{\frac{2}{n}}(n^6x^5+n^4x^3)\int_0^x(v'')^2(z)dz\,dx\\
			&=C \int_{0}^{\frac{2}{n}}(v'')^2(z)\left(\int_{\frac{1}{n}}^{\frac{2}{n}}(n^6x^5+n^4x^3)dx \right )dz  \rightarrow
			0 \quad \text{as}\; n\rightarrow +\infty.
		\end{split}
	\end{equation}
	Analogously the term $\int_{1-\frac{2}{n}}^{1-\frac{1}{n}}(\xi_n''v)^2dx$ tends to $0$ as $n\rightarrow +\infty$.
	Since the remaining terms in \eqref{h1c} can be similarly estimated, 
	one has that
	\[
	\lim_{n \rightarrow +\infty}\int_0^1((v_n- v)'')^2dx=0
	\] 
	and our preliminary claim is proved.
	
	Now, fixed $u\in D(A)$,  set $\Phi(v):=\int_{0}^{1}u''''v\,dx-\int_0^1u''v''dx$, with $v\in H^2_{\frac{1}{a}}(0, 1)$
	$\left(\text{observe that } u''''v\in L^1(0,1) \text{ since }\ds u''''v=\sqrt{a}u''''\frac{v}{\sqrt{a}} \right)$.
	By definition of $\Phi$, it follows that
	\begin{equation*}
		\Phi(v)=0
	\end{equation*}
	for all $v\in H^2_c(0, 1)$. 
	
	In order to prove this fact,  we assume $x_0=0$, the case $x_0=1$ being treated in analogous way.
	Now, let $v\in H^2_c(0, 1)$ and let $\delta >0$ be such that $ \supp v\subset \mathcal K$, where $\mathcal K:=[\delta, 1]$ (or $\mathcal K=[0,1-\delta]$ if $x_0=1$). By definition of $D(A)$, $u''''\in L^2(\mathcal K)$, thus $u'' \in H^2(\mathcal K)$ (by Lemma \ref{derivatadistribuzionale} with $n=p=2$ and $f=u''$) and, in particular, $u \in H^4(\mathcal K)$. Hence, we can integrate by parts, obtaining
	\begin{equation}\label{1}
		\begin{split}
			\int_{0}^{1}u''''v\,dx
			&=\int_{0}^{\delta}u''''v\,dx+\int_{\delta}^{1}u''''v\,dx \\
			&=\int_{0}^{\delta}u''''v\,dx+\int_{\delta}^{1}u''v''dx,
		\end{split}
	\end{equation}
	since $v(\delta)=v'(\delta)=0$.
	Now we prove that
	\begin{equation}\label{4_1}
		\lim_{\delta\to 0}\int_{\delta}^{1}u''v''dx=\int_{0}^{1}u''v''dx
	\end{equation}
	and
	\begin{equation}\label{4}
		\lim_{\delta\to 0}\int_{0}^{\delta}u''''v\,dx=0.
	\end{equation}
	To this aim, observe that
	\[
	\int_{\delta}^{1}u''v''dx=\int_{0}^{1}u''v''dx-\int_{0}^{\delta}u''v''dx.
	\]
	Moreover, as before for $u''''v$, also $u''v''$ belongs to $L^1(0, 1)$ using the H\"older's inequality. Thus, for any $\varepsilon > 0$, by the absolute continuity of the Lebesgue integral, there exists $\delta := \delta(\varepsilon)>0$ such that
	\begin{equation*}
		\biggl | \int_{0}^{\delta}u''v''dx \biggr |\le \biggl | \int_{0}^{\delta}|u''v''|\,dx \biggr |<\varepsilon,
	\end{equation*}
	\begin{equation*}
		\biggl | \int_{0}^{\delta}u''''v\,dx \biggr |\le \biggl | \int_{0}^{\delta}|u''''v|\,dx \biggr |<\varepsilon.
	\end{equation*}
	Taking such a $\delta$ in (\ref{1}), from the arbitrariness of $\varepsilon$, we can deduce
	\begin{equation*}
		\lim_{\delta\to 0}\int_{0}^{\delta}u''''v\,dx=\lim_{\delta\to 0}\int_{0}^{\delta}u''v''dx=0.
	\end{equation*}
	Thus, by the previous equalities and by (\ref{1}), \eqref{4_1} and (\ref{4}), it follows that
	\begin{equation*}
		\int_{0}^{1}u''''v\,dx = \int_{0}^{1}u''v''dx\,\,\,\,\Longleftrightarrow\,\,\,\,\Phi(v)=\int_{0}^{1}(u''''v-u''v'')dx=0,
	\end{equation*}
	for all $v\in H^2_c(0, 1)$. Then, $\Phi$ is a bounded linear functional on $H^2_{\frac{1}{a}}(0, 1)$ such that $\Phi =0$ on $H^2_c(0, 1)$, hence $\Phi=0$ on $H^2_{\frac{1}{a}}(0, 1)$, i.e., (\ref{GF1}) holds.
\end{proof}

Actually, in the weakly degenerate case, one can proceed as done for the case $x_0\in (0,1)$ (see below) making the proof simpler.

As a consequence of Proposition \ref{Proposition Green's Formula} one has the next theorem.
\begin{Theorem}\label{Theorem 2.1}
	Assume Hypothesis \ref{ipo1}.	The operator $(A,D(A))$ defined in \eqref{operator} is self-adjoint and non negative on $L^2_{\frac{1}{a}}(0, 1)$. Hence $-A$ generates a contractive analytic semigroup of angle $\displaystyle\frac{\pi}{2}$ on $L^2_{\frac{1}{a}}(0, 1)$.
\end{Theorem}
\begin{proof}
	Observe that $D(A)$ is dense in $L^2_{\frac{1}{a}}(0, 1)$. In order to show that $A$ is self-adjoint it is sufficient to prove that $A$ is symmetric, non negative and $(I+A)(D(A))= L^2_{\frac{1}{a}}(0,1)$. Indeed, if $A$ is non negative and $I+A$ is surjective on $D(A)$, then $A$ is maximal monotone and in this case $A$ is symmetric if and only if $A$ is self-adjoint.

	\underline{$A$ is symmetric}: thanks to Proposition \ref{Proposition Green's Formula}, for any $u,v\in D(A)$, one has
	\begin{equation*}
		\left\langle v, Au \right\rangle_{L^2_{\frac{1}{a}}(0, 1)}=\int_{0}^{1}\frac{v a u''''}{a}\,dx=\int_{0}^{1}v''''u\,dx=\left\langle Av, u \right\rangle_{L^2_{\frac{1}{a}}(0, 1)}.
	\end{equation*}
	\indent
	\underline{$A$ is non negative}: using again  Proposition \ref{Proposition Green's Formula}, for any $u\in D(A)$
	\begin{equation*}
		\left\langle Au, u \right\rangle_{L^2_{\frac{1}{a}}(0, 1)}=\int_{0}^{1}\frac{a u'''' u}{a}\,dx=\int_{0}^{1}(u'')^2dx\ge 0.
	\end{equation*}
	\indent
	\underline{$I+A$ is surjective}: observe that $H^2_{\frac{1}{a}}(0, 1)$, equipped with the inner product
	\begin{equation*}
		\left\langle u, v \right\rangle_{H^2_{\frac{1}{a}}(0, 1)} := \int_{0}^{1} \biggl (\frac{u v}{a}+u''v''\biggr ) dx\quad \,\,\,\,\,\,\,\forall\; u, v \in H^2_{\frac{1}{a}}(0, 1),
	\end{equation*} 
	is a Hilbert space. Moreover
	\begin{equation*}
		H^2_{\frac{1}{a}}(0, 1) \hookrightarrow L^2_{\frac{1}{a}}(0, 1) \hookrightarrow \left(H^2_{\frac{1}{a}}(0, 1)\right)^*,
	\end{equation*}
	where $\left(H^2_{\frac{1}{a}}(0, 1)\right)^*$ is the dual space of $H^2_{\frac{1}{a}}(0, 1)$ with respect to $L^2_{\frac{1}{a}}(0, 1)$. Indeed, the continuous embedding of $H^2_{\frac{1}{a}}(0, 1)$ in $L^2_{\frac{1}{a}}(0, 1)$ is readily seen. In addition, for any $f\in H^2_{\frac{1}{a}}(0, 1)$ and $\varphi \in L^2_{\frac{1}{a}}(0, 1)$
	\begin{equation*}
		\begin{split}
			\left|\left\langle f, \varphi \right\rangle_{L^2_{\frac{1}{a}}(0, 1)}\right| = \biggl |\int_{0}^{1}\frac{f \varphi}{a}\,dx \biggr |=\biggl |\int_{0}^{1}\frac{f}{\sqrt{a}}\frac{\varphi}{\sqrt{a}}\,dx \biggr |
			&\le \norm{f}_{L^2_{\frac{1}{a}}(0, 1)}\norm{\varphi}_{L^2_{\frac{1}{a}}(0, 1)} \\ 
			&\le \norm{f}_2 \norm{\varphi}_{L^2_{\frac{1}{a}}(0, 1)}.
		\end{split}
	\end{equation*}
	Hence, $L^2_{\frac{1}{a}}(0, 1) \hookrightarrow \left(H^2_{\frac{1}{a}}(0, 1)\right)^*$. Then, $\left(H^2_{\frac{1}{a}}(0, 1)\right)^*$ is the completion of $L^2_{\frac{1}{a}}(0, 1)$ with respect to the norm of $\left(H^2_{\frac{1}{a}}(0, 1)\right)^*$.
	Now, for $f \in L^2_{\frac{1}{a}}(0, 1)$, define the functional $F\in \left(H^2_{\frac{1}{a}}(0, 1)\right)^*$ given by 
	\begin{equation*}
		F(v):= \int_{0}^{1}\frac{f v}{a}\,dx\quad \,\,\,\,\,\,\,\forall \;v\in H^2_{\frac{1}{a}}(0, 1).
	\end{equation*}
	Consequently, by the Lax-Milgram Theorem, there exists a unique $u\in H^2_{\frac{1}{a}}(0, 1)$ such that for all $v\in H^2_{\frac{1}{a}}(0, 1)$:
	\begin{equation}\label{Consequence Riesz's Theorem}
		\left\langle u, v \right\rangle_{H^2_{\frac{1}{a}}(0, 1)}=\int_{0}^{1} \frac{f v}{a}\,dx.
	\end{equation}
	In particular, since $\mathcal{C}^{\infty}_c(0, 1) \subset H^2_{\frac{1}{a}}(0, 1)$, the integral representation (\ref{Consequence Riesz's Theorem}) holds for all $v\in \mathcal{C}^{\infty}_c(0, 1)$, i.e.,
	\begin{equation*}
		\int_{0}^{1} u''v''dx = \int_{0}^{1} \frac{f-u}{a}\,v\,dx\quad \,\,\,\,\,\,\,\forall	\;v\;\in \mathcal{C}^{\infty}_c(0, 1).
	\end{equation*}
	Thus, the distributional second derivative of $u''$ is equal to $\ds\frac{f-u}{a}$ a.e. in $(0,1).$ Since $f-u\in L^2_{\frac{1}{a}}(0, 1)$ and
	\begin{equation*}
		au''''=f -u\quad \text{a.e. in } (0,1),
	\end{equation*}
	$u \in D(A)$.
	Hence $(I+A)(u)=f$. As an immediate consequence of the Stone-von Neumann Spectral Theorem and functional calculus associated with the Spectral Theorem, one has that the operator $(A,D(A))$ generates a cosine family and an analytic semigroup of angle $\displaystyle\frac{\pi}{2}$ on $L^2_{\frac{1}{a}}(0, 1)$.
\end{proof}

The previous result can be used to prove that the one-dimensional fourth order parabolic systems associated to the operator $Au:= au''''$ are well posed. 
More precisely, fixed $T>0$, $u_0\in L^2_{\frac{1}{a}}(0,1)$ and $f\in L^2(0,T;L^2_{\frac{1}{a}}(0,1))$, the problem
\begin{equation}\label{P}
	\begin{cases}
		\ds \frac{\partial u}{\partial t}(t,x)+a(x)\frac{\partial^4u}{\partial x^4}(t,x)=f(t,x), &(t,x)\in (0,T) \times (0,1),\\
		u(t,0)=u(t,1)=0,& t\in (0,T),\\
		\ds \frac{\partial u}{\partial x}(t,0)=\frac{\partial u}{\partial x}(t,1)=0,&t\in (0,T),\\
		u(0,x)=u_0(x),&x\in(0,1),
	\end{cases}
\end{equation}
admits a unique solution $u\in \mathcal{C}([0,T];L^2_{\frac{1}{a}}(0,1))\cap L^2(0,T;H^2_{\frac{1}{a}}(0,1))$ (see also \cite{4}).  As we will see, the same result holds in the case $x_0 \in (0,1)$.
These preliminary considerations will be the starting point to study the observability and the null controllability for this kind of problems.

Observe that in the non degenerate case, i.e., if $a(x) >0$ for all $x\in [0,1]$, the previous problem is also known as Cahn–Hilliard type problem.

\section{The domain of the operator $A$ and the case $x_0 \in (0,1)$} \label{Section 3}
Assuming further hypotheses on the function $a$, we can prove some characterizations for $D(A)$ in both the weak and the strong case. More precisely, we assume the following:
\begin{Assumptions}[Weakly Degenerate Function] \label{hp 3.1}
	The function $a\in \mathcal C[0,1]$ is weakly degenerate, i.e., there exists a point $x_0\in [0,1]$ such that  $a(x_0)=0$, $a>0$ on
	$[0,1]\setminus \{x_0\}$ and $\frac{1}{a}\in L^1(0,1)$.
\end{Assumptions}
For example, as $a$, we can consider $a(x)=|x-x_0|^{\alpha}$, $0<\alpha<1$.
\begin{Assumptions}[Strongly Degenerate Function] \label{hp 3.2}
	The function $a\in \mathcal{C}[0, 1]$ is strongly degenerate, i.e., there exists a point $x_0\in [0,1]$ such that  $a(x_0)=0$, $a>0$ on
	$[0,1]\setminus \{x_0\}$ and $\frac{1}{a}\notin L^1(0,1)$.
\end{Assumptions}
For example, as $a$, we can consider $a(x)=|x-x_0|^{\alpha}$, $\alpha\ge 1$.
\newline
\indent
Thanks to the previous assumptions, we can characterize the spaces introduced in Section \ref{SECTION 2}. In particular, we have the following results.
\begin{Proposition} \label{proposition 3.1}
	Assume Hypothesis \ref{hp 3.1}. Then, the spaces $H^i_{\frac{1}{a}}(0, 1)$ 
	and $H^i_0(0, 1)$, $i=1,2$, coincide algebraically and the two norms are equivalent. Moreover, if $u \in H^2_0(0,1)$, i.e., $u \in H^2_{\frac{1}{a}}(0, 1)$, then $(au)(x_0)=(au')(x_0)=0$.
\end{Proposition}
\begin{proof}
	By definition $H^i_{\frac{1}{a}}(0, 1)\subseteq H^i_0(0, 1)$, $i=1,2$. Moreover, if $u\in H^i_0(0, 1)$ then $u \in \mathcal C[0,1]$ and, using the fact that $\frac{1}{a}\in L^1(0,1)$, one has
	\begin{equation*}
		\int_{0}^{1}\frac{u^2}{a}\,dx \le \max_{[0,1]}u^2 \int_{0}^{1}\frac{1}{a}\,dx\in\mathbb{R},
	\end{equation*}
	i.e., $u\in L^2_{\frac{1}{a}}(0, 1)$. This implies $u\in H^i_{\frac{1}{a}}(0, 1)$.
	\newline
	\indent
	Furthermore, since $H^i_0(0, 1)$ is continuously embedded in $\mathcal{C}[0,1]$, one has that for all $u\in H^i_0(0, 1)$
	\begin{equation*}
		\|u\|_{L^2_{\frac{1}{a}}(0,1)}^2=\int_{0}^{1}\frac{u^2}{a}\,dx \le \Vert u \Vert^2_{\mathcal{C}[0,1]} \norm{\frac{1}{a}}_{L^1(0,1)}\le C \Vert u \Vert^2_{H^i_0(0, 1)},
	\end{equation*}
	where $C$ is a positive constant. In this way, for all $u \in H^i_0(0, 1)$,
	\begin{equation*}
		\Vert u \Vert_i \le (C+1) \Vert u \Vert_{H^i_0(0, 1)} \le (C+1) \Vert u \Vert_i,
	\end{equation*}
	for a positive constant $C$. 
	
	Now, take $u \in H^2_0(0,1)$. Hence $u, u' \in \mathcal C[0,1]$  and, using the fact that $a(x_0)=0$, one has $(au)(x_0)=(au')(x_0)=0$.
\end{proof}

Thus,  the space $\mathcal{C}^{\infty}_c(0, 1)$ is dense in $H^i_{\frac{1}{a}}(0, 1)$, $i=1,2$.
Observe that, if $x_0 \in \{0,1\}$, the conditions $(au)(x_0)=(au')(x_0)=0$ are clearly satisfied and the domain $D(A)$ defined in \eqref{operator} can be rewritten as
\[
D(A)=\left\{u\in H^2_0(0,1): \, au''''\in L^2_{\frac{1}{a}}(0, 1) \right\}.
\]

If $x_0 \in (0,1)$ and Hypothesis \ref{hp 3.1} is satisfied, we consider as $D(A)$ the same domain given in (\ref{operator}), i.e.,
\[
D(A):=\left\{u\in H^2_{\frac{1}{a}}(0,1): \, au''''\in L^2_{\frac{1}{a}}(0, 1)\right\}.\]
Clearly, thanks to Proposition \ref{proposition 3.1}, it can be rewritten as
\[
\begin{split}
	D(A)=\Bigl \{u\in H^2_0(0,1):\, &(au)(x_0)= (au')(x_0)=0, \, au''''\in L^2_{\frac{1}{a}}(0, 1) \Bigr\}.
\end{split}
\]
In this case \eqref{GF1} follows immediately for all $(u,v)\in D(A)\times H^2_{\frac{1}{a}}(0, 1)$. Indeed, if $au''''\in L^2_{\frac{1}{a}}(0, 1)$, then $u''''\in L^1(0,1)$ and $u\in W^{4,1}(0,1)\hookrightarrow \mathcal{C}^3[0,1]$. Hence, under Hypothesis \ref{hp 3.1}, Theorem \ref{Theorem 2.1} still holds and problem \eqref{P} admits a unique solution $u$.

In the strongly degenerate case, if $x_0 \in (0,1)$, the domain of $A$ is given again by \eqref{operator}, but in order to characterize it, we introduce the following space
\begin{equation*}
	X:=\biggl \{u\in H^2_{\frac{1}{a}}(0, 1): u(x_0)=(au')(x_0)=0 \biggr \}.
\end{equation*}
Notice that, when $x_0 \in \{0,1\}$, then 
\begin{equation}\label{eq1}
	X= H^2_{\frac{1}{a}}(0, 1)
\end{equation}
in a trivial way.
\

Actually, the equality \eqref{eq1} can also be proved if $x_0\in (0,1)$ adding a further assumption on the function $a$, as we will see in Proposition \ref{Prop 3.2}.
\begin{Assumptions}\label{hp 3.3}
	Assume that
	there exist $K \in [1,2]$ and $C>0$ such that 
	\[
	\frac{1}{a(x)}\le \frac{C}{|x-x_0|^K}
	\]
	for all $x \in [0,1]\setminus\{x_0\}$.
\end{Assumptions}
Observe that the previous hypothesis is obviously satisfied by the prototype $|x-x_0|^K$, where $K\in [1,2]$. 

Proceeding as in \cite[Lemma 3.7]{3}, if $x_0\in (0,1)$, or as in \cite[Proposition 2.6]{2}, if $x_0=0$ or $x_0=1$, one can prove the following result.
\begin{Lemma}\label{lemma norme equivalenti}
	Assume Hypotheses \ref{hp 3.2} and \ref{hp 3.3}. Then there exists a positive constant $C$ such that
	\begin{equation*}
		\int_{0}^{1}\frac{v^2}{a}\,dx \le C \int_{0}^{1}(v')^2dx  \quad\,\,\,\,\,\,\, \forall \;v\in X.
	\end{equation*}
\end{Lemma}
Actually Lemma \ref{lemma norme equivalenti} is proved in \cite[Proposition 2.6]{2} if $x_0=0$ or $x_0=1$ under a stronger assumption, but it is evident by the proof that it holds under Hypothesis \ref{hp 3.3}.
\begin{Proposition} \label{Prop 3.2}
	If Hypotheses \ref{hp 3.2} and \ref{hp 3.3} are satisfied and $x_0\in (0,1)$, then (\ref{eq1}) holds and the norms $\Vert u \Vert_{2}^2$ and $\int_{0}^{1}(u'')^2(x)dx$ are equivalent.
\end{Proposition}
\begin{proof} 
	Obviously, $X \subseteq H^2_{\frac{1}{a}}(0, 1)$. Now, take $u\in H^2_{\frac{1}{a}}(0, 1)$. In order to have the thesis, it is sufficient to prove that
	$u(x_0)=(au')(x_0)=0$. Indeed,
	since $u \in H^2_0(0,1)$, then $u$, $u'$, and hence $au'$, belong to $\mathcal C[0,1]$. This implies that there exists
	\begin{equation*}
		\lim_{x \to x_0} u(x)= u(x_0)=L\in\mathbb{R}.
	\end{equation*}
	Clearly, $L=0$. Indeed, if $L\neq 0$, then there exists $C>0$ such that
	\begin{equation*}
		|u(x)|\ge C,
	\end{equation*}
	for all $x$ in a neighbourhood of $x_0$, $x\neq x_0$. Thus,
	\begin{equation*}
		\frac{u^2(x)}{a(x)} \ge \frac{C^2}{a(x)},
	\end{equation*}
	for all $x$ in a neighbourhood of $x_0$, $x\neq x_0$. Since, by hypothesis, $\frac{1}{a}\notin L^1(0,1)$, we obtain $u \notin L^2_{\frac{1}{a}}(0, 1)$. Hence $L=0$.
	Moreover, 	arguing as before, there exists
	\begin{equation*}
		\lim_{x \to x_0} u'(x)= u'(x_0)=M\in\mathbb{R}
	\end{equation*}	
	and hence, using the fact that $a(x_0)=0$, one has
	\[
	\lim_{x \to x_0} (au')(x)=0.
	\]

	Now, we prove that the two norms are equivalent. To this aim, take $u \in X$; thus, by Lemma \ref{lemma norme equivalenti} and using the fact that $u'(0)=u'(1)=0$, one can prove that there exists a positive constant $C$ such that
	\begin{equation*}
		\Vert u'' \Vert^2_{L^2(0,1)} \le \Vert u \Vert^2_{2} \le C( \Vert u' \Vert^2_{L^2(0,1)}+\Vert u'' \Vert^2_{L^2(0,1)}) \le  C \Vert u'' \Vert^2_{L^2(0,1)}.
	\end{equation*}
	Hence the two norms are equivalent.
\end{proof}

As a consequence of the previous proposition one has that if Hypotheses \ref{hp 3.2} and \ref{hp 3.3} are satisfied and $x_0\in (0,1)$, then the domain of the operator can be written as
\[
D(A):=\left\{ u \in H^2_{\frac{1}{a}}(0,1):  u(x_0)=(au')(x_0)=0 \text{ and } au''''\in L^2_{\frac{1}{a}}(0,1)\right\}.
\]
Hence, proceeding as in the proof of Proposition \ref{Proposition Green's Formula}, one can prove that \eqref{GF1} is satisfied. Indeed, in this case, in order to prove that $H^2_c(0,1)$ is dense in $H^2_{\frac{1}{a}}(0,1)$ we have to modify the definition of $\xi_n$ in a suitable way (see also \cite{3}). Then, one can consider $v\in H^2_c(0, 1)$ and take $\delta >0$ such that $ \supp v\subset \mathcal K$, where $\mathcal K:= [0, x_0-\delta]\cup[x_0+ \delta, 1]$. By definition of $D(A)$, $u''''\in L^2(\mathcal K)$, thus $u'' \in H^2(\mathcal K)$ (by Lemma \ref{derivatadistribuzionale} with $n=p=2$ and $f=u''$) and, in particular, $u \in H^4(\mathcal K)$. Hence, we can integrate by parts, obtaining
\[
\begin{split}
	\int_{0}^{1}u''''v\,dx
	&=\int_{0}^{x_0-\delta}u''''v\,dx+\int_{x_0-\delta}^{x_0+\delta}u''''v\,dx+\int_{x_0+\delta}^{1}u''''v\,dx \\
	&=\int_{0}^{x_0-\delta}u''v''dx+\int_{x_0-\delta}^{x_0+\delta}u''''v\,dx+\int_{x_0+\delta}^{1}u''v''dx,
\end{split}
\]
since $v\in H^2_c(0, 1)$. The rest of the proof is similar to the case $x_0\in \{0,1\}$, so we omit it. Hence, under Hypotheses \ref{hp 3.2} and \ref{hp 3.3}, Theorem \ref{Theorem 2.1} still holds and problem \eqref{P} admits a unique solution $u$.

Now, we come back to Proposition \ref{Prop 3.2}, observing that if we know a priori that $u' \in L^2_{\frac{1}{a}}(0,1)$, then we could prove that $u'(x_0)=0$. Indeed the next result holds.
\begin{Proposition} \label{Prop 3.2'}
	Assume Hypotheses \ref{hp 3.2} and \ref{hp 3.3} and $x_0\in (0,1)$. If $u \in  H^2_{\frac{1}{a}}(0, 1)$ is such that $u' \in L^2_{\frac{1}{a}}(0,1)$, then $u'(x_0)=0$. 
\end{Proposition}
\begin{proof} 
	Let $u \in H^2_{\frac{1}{a}}(0, 1)$ so that $u'\in L^2_{\frac{1}{a}}(0,1)$.  In order to obtain the thesis, one can proceed as in Proposition \ref{Prop 3.2}. Indeed, as before,
	$u' \in \mathcal C[0,1]$; this implies that there exists
	\begin{equation*}
		\lim_{x \to x_0} u'(x)= u'(x_0)=M\in\mathbb{R}.
	\end{equation*}
	Clearly, $M=0$. Indeed, if $M\neq 0$, then there exists $C>0$ such that
	\begin{equation*}
		|u'(x)|\ge C,
	\end{equation*}
	for all $x$ in a neighbourhood of $x_0$, $x\neq x_0$. Thus,
	\begin{equation*}
		\frac{(u')^2(x)}{a(x)} \ge \frac{C^2}{a(x)},
	\end{equation*}
	for all $x$ in a neighbourhood of $x_0$, $x\neq x_0$. As in Proposition \ref{Prop 3.2}, one can conclude that $M=0$.
\end{proof}

Observe that, if $x_0\in \{0,1\}$ and $u\in H^2_{\frac{1}{a}}(0, 1)$, then $u'(x_0)=0$ without additional assumptions on $a$ or on $u$.

Thanks to the previous proposition, one can prove the next result. To this purpose,
define
\[
\widetilde H^2_{\frac{1}{a}}(0, 1): = \left\{u\in H^2_{\frac{1}{a}}(0, 1): u' \in L^2_{\frac{1}{a}}(0,1)\right\}
\]
and
\[
\widetilde X:=\biggl \{u\in H^2_{\frac{1}{a}}(0, 1): u(x_0)=u'(x_0)=0 \biggr \}.
\]

\begin{Proposition} \label{Prop 3.2''}
	If Hypotheses \ref{hp 3.2} and \ref{hp 3.3} are satisfied, then there exists a positive constant $C$ such that
	\begin{equation}\label{primadis}
		\int_{0}^{1}\frac{(u')^2}{a}\,dx \le C \int_{0}^{1}(u'')^2dx,
	\end{equation}
	for all $u \in \widetilde X$.
	Hence 
	\begin{equation}\label{prima=}
		\widetilde H^2_{\frac{1}{a}}(0, 1)=\widetilde X.
	\end{equation}
\end{Proposition}
\begin{proof} Take $u \in \widetilde X$ and assume, first of all, that $x_0 \in \{0,1\}$. Then \eqref{primadis} follows by \cite[Lemma 3.7]{2} applied to $w:= u' \in H^1_0(0,1)$.
	
	Now, assume $x_0 \in (0,1)$ and define again $w:= u' \in H^1_0(0,1)$. Clearly, $w(x_0)=0$.
	By Hypothesis \ref{hp 3.3}, there exists $C>0$ such that
	$\displaystyle\frac{1}{a(x)} \le \frac{C}{|x-x_0|^2}$ for all $x
	\in [0,1]\setminus \{x_0\}$.
	Then, for  a suitable
	$\varepsilon>0$ and using the assumptions on $a$ and the Hardy's inequality, one has \[
	\begin{aligned}
		\int_0^1 w^2\frac{1}{a} dx &=\int_{0}^{x_0-\epsilon}w^2\frac{1}{a} dx
		+ \int_{x_0-\epsilon}^{x_0+ \epsilon} w^2\frac{1}{a} dx +
		\int_{x_0+
			\epsilon}^{1} w^2\frac{1}{a} dx
		\\
		&\le
		\frac{1}{\min_{[0, x_0-\epsilon]}a}\int_0^{x_0-\epsilon} w^2dx
		+
		\int_{x_0-\epsilon}^{x_0+\epsilon} w^2\frac{1}{a}dx
		\\& +
		\frac{1}{\min_{[x_0+\epsilon,1]}a}\int_{x_0 +\epsilon}^1 w^2dx
		\\
		& \le
		\frac{1}{\min_{[0, x_0-\epsilon]\cup [x_0+\epsilon,1]}a}\int_0^1 w^2 dx
		+ \int_{x_0-\epsilon}^{x_0}w^2\frac{1}{a} dx  + \int_{x_0}^{x_0+ \epsilon}w^2\frac{1}{a} dx
		\\
		& \le
		\frac{1}{\min_{[0, x_0-\epsilon]\cup [x_0+\epsilon,1]}a}\int_0^1 w^2 dx
		+ C\int_{x_0-\epsilon}^{x_0}w^2\frac{1}{|x-x_0|^2} dx \\
		&+C\int_{x_0}^{x_0+ \epsilon}w^2\frac{1}{|x-x_0|^2} dx \\
		& \le
		\frac{1}{\min_{[0, x_0-\epsilon]\cup [x_0+\epsilon,1]}a}\int_0^1 w^2 dx
		+ C_H\!\int_{x_0-\epsilon}^{x_0}\!\!(w')^2 dx \\
		& +C_H\!\int_{x_0}^{x_0+
			\epsilon}\!\!(w')^2dx \\
		& \le
		C \left( \int_0^{1}w^2 dx  + \int_0^1 (w')^2 dx\right),
	\end{aligned}
	\]
	for a positive constant $C$; here $C_H$ is the Hardy constant.
	By Poincar\'{e}'s inequality, it follows that
	\[
	\int_0^1 w^2\frac{1}{a} dx\le C \int_0^1 (w')^2dx
	\]
	for a suitable constant $C$. Hence \eqref{primadis} holds. This implies that $\widetilde X \subseteq \widetilde H^2_{\frac{1}{a}}(0,1)$. By Proposition \ref{Prop 3.2'} the other inclusion follows immediately, hence \eqref{prima=} holds. 
\end{proof}

\begin{Remark}\label{rem}
	Clearly, if  Hypotheses \ref{hp 3.2} and \ref{hp 3.3} are satisfied and  $u\in H^2_{\frac{1}{a}}(0, 1)$ satisfies \eqref{primadis}, then $u' \in L^2_{\frac{1}{a}}(0,1)$ and again $u'(x_0)=0$ by Proposition  \ref{Prop 3.2''}.
\end{Remark}
Adding an additional assumption on the function $a$, one can prove a characterization on $D(A)$.
\begin{Assumptions}\label{hp 3.3.2}
	Assume that the function $a$ belongs to $W^{1, \infty}(0,1)$ and is
	\[
	\!\!\!\!\!\!\!\!\!\!\!\!\!\!\!\!\begin{cases}\text{non increasing on the left and non decreasing on the right of } x_0, &\!\!\!\!\!\!\text{ if } x_0 \in (0,1),\\
		\text{non decreasing on the right of } x_0, &\!\!\!\!\!\!\text{ if } x_0=0,\\
		\text{non increasing on the left of } x_0, &\!\!\!\!\!\!\text{ if } x_0=1.
	\end{cases}\]
\end{Assumptions}
\begin{Proposition} 	
	Let
	\begin{equation*}
		D:=\biggl \{u \in H^2_{\frac{1}{a}}(0, 1):\, au''''\in L^2_{\frac{1}{a}}(0, 1), \; u(x_0)=(au')(x_0)=(au'')(x_0)=0 \biggr \}.
	\end{equation*}
	If Hypotheses \ref{hp 3.2}, \ref{hp 3.3} and \ref{hp 3.3.2} are satisfied, then $D(A) =D$.
\end{Proposition}
\begin{proof}
	Evidently $D\subseteq D(A)$. Now, we take $u \in D(A)$ and we prove that $u \in D$. By Proposition \ref{Prop 3.2}, $u(x_0)=(au')(x_0)=0$. Thus, it is sufficient to prove that $(au'')(x_0)=0$ (observe that, since $u \in H^2(0,1)$ and $a \in W^{1,\infty}(0, 1)$, then $u'' \in L^2(0,1)$, $\sqrt{a}u'' \in L^2(0,1)$ and $au'' \in L^2(0,1)$).
	
	Let $\delta >0$ sufficiently small and $x=x_0-\delta$. From the proof of Theorem \ref{Theorem 2.1} it follows that $u\in H^{4}(\mathcal  K)$, where $\mathcal  K:=[0,x]$. In particular, since $au''$ is a continuous function in $\mathcal  K$, the following formula holds:
	\begin{equation}\label{second term}
		(au'')(x)-(au'')(0)=\int_{0}^{x}(au'')'(t)dt=\int_{0}^{x}(a'u'')(t)dt+\int_{0}^{x}(au''')(t)dt.
	\end{equation}
	Now we will estimate the last two terms in (\ref{second term}). To this aim, observe that
	\begin{equation*}
		\int_{0}^{x}(a'u'')(t)dt = \int_{0}^{x_0}(a'u'')(t)dt-\int_{x}^{x_0}(a'u'')(t)dt
	\end{equation*}
	and
	\begin{equation}\label{formula (3.7)}
		\begin{split}
			\int_{0}^{x}(au''')(t)dt
			&=\int_{0}^{x}a(t)\biggl (\int_{0}^{t}u''''(s)ds+u'''(0)\biggr )dt\\
			&= \int_{0}^{x}a(t)\int_{0}^{t}u''''(s)ds dt+\int_{0}^{x}a(t)u'''(0)dt.
		\end{split}
	\end{equation}
	In particular, we have
	\[
	\lim_{x \to x_0}\int_{0}^{x}(a'u'')(t)dt = \int_{0}^{x_0}(a'u'')(t)dt.
	\]
	Indeed, Hypothesis \ref{hp 3.2} and  $u''\in L^2(\mathcal  K)\subset L^1(\mathcal  K)$ imply that
	\[
	\left| \int_{x}^{x_0}(a'u'')(t)dt \right| \le \int_{x}^{x_0}|(a'u'')(t)|dt\le \Vert a'\Vert_{L^\infty(0,1)}\int_{x}^{x_0}|u''(t)|dt.
	\]
	Hence
	\begin{equation*}
		\lim_{x \to x_0}\int_{x}^{x_0}(a'u'')(t)dt = 0
	\end{equation*}
	by the absolute continuity of the Lebesgue integral. As far as the two terms in (\ref{formula (3.7)}) are concerned, we have
	\begin{equation*}
		\lim_{x \to x_0}\int_{0}^{x}a(t)u'''(0)dt=u'''(0)\lim_{x \to x_0}\int_{0}^{x}a(t)dt = u'''(0)\int_{0}^{x_0}a(t)dt,
	\end{equation*}
	where we recall that $\int_{0}^{x}a(t)\,dt$ can be written as $\int_{0}^{x_0}a(t)\,dt-\int_{x}^{x_0}a(t)\,dt$ and, thanks to the hypotheses on $a$, $\lim_{x \to x_0}\int_{x}^{x_0}a(t)\,dt = 0$. Moreover, using the monotonicity condition of $a$ on the left of $x_0$, one has
	\[
	\begin{aligned}
		\left|a(t)\int_{0}^{t}u''''(s)ds\right| &=\left|\sqrt{a(t)}\int_{0}^{t}\sqrt{a(t)}u''''(s)       ds\right|\\
		&\le\sqrt{a(t)}\int_{0}^{t}\sqrt{a(s)}|u''''(s)  |     ds\\
	\end{aligned}
	\]
	for all $t\in (0,x)$. Using again the assumptions on $a$ and the fact that $au''''\in L_{\frac{1}{a}}^2(0,1)$, one can conclude that
	\[
	f(t) :=a(t)\int_{0}^{t}u''''(s) ds \in L^1(0,x_0).
	\]
	Arguing as before,
	
		\begin{equation*}
			\lim_{x \to x_0}\int_{0}^{x}f(t)dt=\int_{0}^{x_0}f(t)dt
		\end{equation*}
		and hence
		\begin{equation*}
			\exists\lim_{x \to x^-_0}(au'')(x)=(au'')(x^-_0)=L\in\mathbb{R}.
		\end{equation*}
		In a similar way, one can prove that
		\begin{equation*}
			\exists\lim_{x \to x^+_0}(au'')(x)=(au'')(x^+_0)=M\in\mathbb{R}.
		\end{equation*}
		In order to complete the proof it remains to prove that $L=M=0$. Indeed, if $L\neq 0$, then there exist $C>0$ and a left neighbourhood $I_-$ of $x_0$ such that 
		\begin{equation*}
			a(x)|u''(x)| \ge C \quad\,\,\,\,\,\,\,\,\, \forall\, x \in I_-.
		\end{equation*}
		Thus,
		\begin{equation*}
			|u''(x)| \ge \frac{C}{a(x)}\quad\,\,\,\,\,\,\,\, \forall \, x \in I_-.
		\end{equation*}
		But $\frac{1}{a}\notin L^1(0,1)$, thus $u'' \notin L^1(0,1)$ and this contradicts the fact that $u''\in L^2(0,1)\subset L^1(0,1)$. Hence $L=0$; analogously one can prove $M=0$. Thus, we obtain $L=M=0$. In this way $\lim_{x \to x_0}(au'')(x)$ exists and it is equal to $0$, i.e., $(au'')(x_0)=0$. 
	\end{proof}

	\section{The general operator of order $2n$}\label{Section4}
	In this section we will extend the previous results to a general operator $A_nu:= au^{(2n)}$. To this aim, taking $n \in \N$ with $n \ge 3$ (the case $n=2$ is considered in the previous sections), we introduce the following spaces
	
	\[
	H^i_{\frac{1}{a}}(0, 1):= L^2_{\frac{1}{a}}(0, 1)\cap H^i_0(0, 1),
	\]
	with the norm
	\begin{equation*}
		\norm{u}_{H^i_{\frac{1}{a}}(0, 1)}^2:=\norm{u}^2_{L^2_{\frac{1}{a}}(0, 1)} + \sum_{k=1}^{i}\Vert u^{(k)}\Vert^2_{L^2(0, 1)}\quad\,\,\,\,\,\,\, \forall \; u \in H^i_{\frac{1}{a}}(0, 1),
	\end{equation*}
	$i=3,...., n$.
	
	As before, $H^i_0(0, 1):= \{u\in H^{i}(0, 1): u^{(k)}(j)=0,\, j=0,1,\, k=0,1, ..., i-1\}$, $i=3,...,n$, and for all $u \in H^i_{\frac{1}{a}}(0, 1)$ the norms $\norm{u}_{H^i_{\frac{1}{a}}(0, 1)}^2$ and
	\[
	\|u\|_i^2:=\norm{u}^2_{L^2_{\frac{1}{a}}(0, 1)}+ \|u^{(i)}\|^2_{L^2(0, 1)}
	\]
	are equivalent; hence in the following we will use $\|\cdot\|_i$.
	
	Define the operator $A_n$ by 
	\begin{equation*}
		A_nu:=au^{(2n)}, \quad \forall \; u \in
		D(A_n):=\biggl \{u\in H^n_{\frac{1}{a}}(0, 1):\, au^{(2n)}\in L^2_{\frac{1}{a}}(0, 1) \biggr \},
	\end{equation*}
	if $x_0 \in \{0,1\}$, or in the strongly degenerate case when $x_0 \in (0,1)$. 
	The next general Green's formula holds.
	\begin{Proposition}[Green's Formula]\label{Proposition Green's Formula2n}
		Assume Hypothesis \ref{ipo1}	 if $x_0 \in \{0,1\}$ or Hypotheses \ref{hp 3.2} and \ref{hp 3.3} if $x_0 \in (0,1)$. Then, for all $n \in \N$ with $n \ge  3$ and for all $(u,v)\in D(A_n)\times H^n_{\frac{1}{a}}(0, 1)$ one has
		\begin{equation*}
			\int_{0}^{1}u^{(2n)}v\,dx=(-1)^n\int_{0}^{1}u^{(n)}v^{(n)}dx.
		\end{equation*}
	\end{Proposition}
	
	\begin{proof}
		The proof is similar to the one of Proposition \ref{Proposition Green's Formula}, so we sketch it. Actually, it is sufficient to prove that, fixed $u\in D(A_n)$ and  defined $\Phi(v):=\int_{0}^{1}u^{(2n)}v\,dx +(-1)^{n+1}\int_0^1 u^{(n)}v^{(n)}dx$, with $v\in H^n_{\frac{1}{a}}(0, 1)$, it follows that
		\begin{equation}\label{Phi}
			\Phi(v)=0
		\end{equation}
		for all  $v\in H^n_{\frac{1}{a}}(0,1)$. As in Proposition \ref{Proposition Green's Formula}, it is sufficient to prove \eqref{Phi}
		for all $v\in H^n_c(0, 1):= \{v \in H^n(0,1): \text{supp}\, v \subset (0,1) \setminus \{x_0\}\}$. The thesis will follow using the density of $H^n_c(0, 1)$ in $H^n_{\frac{1}{a}}(0,1)$, which can be proved as in  Proposition \ref{Proposition Green's Formula}.

		Indeed, assume $x_0 \in (0,1)$, let $\delta >0$ and set $\mathcal K:= [0, x_0-\delta]\cup [x_0+ \delta, 1]$.  By definition of $D(A_n)$, $u^{(2n)}\in L^2(\mathcal  K)$, thus $u^{(n)} \in H^n(\mathcal  K)$ (by Lemma \ref{derivatadistribuzionale} with $p=2$) and, in particular, $u \in H^{2n}(\mathcal  K)$. 
		Observe that, fixed $v\in H^n_c(0, 1)$, one can easily prove that
		\begin{equation*}
			\int_{\mathcal  K}u^{(2n)}v\,dx= (-1)^{n}\int_{\mathcal  K}u^{(n)}v^{(n)}dx.
		\end{equation*}
		Hence
		\begin{equation}\label{1'}
			\begin{split}
				&\int_{0}^{1}u^{(2n)}v\,dx
				=\int_{0}^{x_0-\delta}u^{(2n)}v\,dx+\int_{x_0-\delta}^{x_0+\delta}u^{(2n)}v\,dx+\int_{x_0+\delta}^{1}u^{(2n)}v\,dx \\
				&=(-1)^{n}\int_{0}^{x_0-\delta}u^{(n)}v^{(n)}dx+\int_{x_0-\delta}^{x_0+\delta}u^{(2n)}v\,dx+(-1)^{n}\int_{x_0+\delta}^{1}u^{(n)}v^{(n)}dx.
			\end{split}
		\end{equation}
		As in \eqref{4_1} and in \eqref{4}, one can prove
		\[
		\lim_{\delta\to 0}\int_{0}^{x_0-\delta}u^{(n)}v^{(n)}dx=\int_{0}^{x_0}u^{(n)}v^{(n)}dx,
		\]
		\[
		\lim_{\delta\to 0}\int_{x_0+\delta}^{1}u^{(n)}v^{(n)}dx=\int_{x_0}^{1}u^{(n)}v^{(n)}dx
		\]
		and
		\[
		\lim_{\delta\to 0}\int_{x_0-\delta}^{x_0+\delta}u^{(2n)}v\,dx=0.
		\]
		Thus, by the previous limits and \eqref{1'}, it follows
		\[
		\begin{aligned}
			&	\int_{0}^{1}u^{(2n)}v\,dx = (-1)^{n}\int_{0}^{1}u^{(n)}v^{(n)}dx\,\,\,\,\Longleftrightarrow\\
			&\Phi(v)=\int_{0}^{1}(u^{(2n)}v+ (-1)^{n+1}u^{(n)}v^{(n)})dx=0,
		\end{aligned}
		\]
		for all $v\in H^n_c(0, 1)$. For the rest of the proof, one can proceed as in Proposition \ref{Proposition Green's Formula}. The case $x_0 \in \{0,1\}$ is similar, so we omit it.
	\end{proof}
	
	In the weakly degenerate case, if $x_0 \in (0,1)$, we consider the same domain as before
	\[
	A_nu:=au^{(2n)}, \; \forall \; u \in
	D(A_n):=\biggl \{u\in H^n_{\frac{1}{a}}(0, 1):\, au^{(2n)}\in L^2_{\frac{1}{a}}(0, 1)  \biggr \},
	\]
	and we underline that, as for the case $n=2$, if $au^{(2n)}\in L^2_{\frac{1}{a}}(0, 1)$, then $u\in W^{2n,1}(0,1)$. Thus, the Proposition \ref{Proposition Green's Formula2n} follows immediately.
	
	A natural consequence of the previous proposition is collected in the next theorem, which establishes the generation property in the general case.
	\begin{Theorem}
		If $x_0 \in \{0,1\}$, assume Hypothesis \ref{ipo1}; if $x_0 \in (0,1)$ assume  Hypothesis \ref{hp 3.1} or Hypotheses \ref{hp 3.2} and \ref{hp 3.3}. Then, the operator $\tilde{A_n}:D(A_n)\to L^2_{\frac{1}{a}}(0, 1)$, defined by $\tilde{A_n}u:=(-1)^{n}au^{(2n)}$ for all $u \in D(A_n)$, is self-adjoint and non negative on $L^2_{\frac{1}{a}}(0, 1)$. Hence $-\tilde{A_n}$ generates a contractive analytic semigroup of angle $\displaystyle\frac{\pi}{2}$ on $L^2_{\frac{1}{a}}(0, 1)$ for any $n \in \mathbb{N}$, $n \ge 3$.
	\end{Theorem}
	\begin{proof}
		Observe that $D(A_n)$ is dense in $L^2_{\frac{1}{a}}(0, 1)$. It is also clear that
		$\tilde{A_n}$ is symmetric and non negative.
		Indeed,
		
		\underline{$\tilde{A_n}$ is symmetric}: for any $u,v\in D(A_n)$
		\begin{equation*}
			\begin{split}
				\langle \tilde{A_n}u, v \rangle_{L^2_{\frac{1}{a}}(0, 1)}&=\int_{0}^{1}\frac{(-1)^{n}au^{(2n)} v }{a}\,dx=\int_{0}^{1}u^{(n)}v^{(n)}dx=\int_{0}^{1}\frac{(-1)^{n}av^{(2n)} u }{a}\,dx  \\
				&=\langle u, \tilde{A_n}v \rangle_{L^2_{\frac{1}{a}}(0, 1)}.
			\end{split}
		\end{equation*}
		\indent
		\underline{$\tilde{A_n}$ is non negative}: for any $u\in D(A_n)$
		\begin{equation*}
			\langle \tilde{A_n}u, u \rangle_{L^2_{\frac{1}{a}}(0, 1)}=\int_{0}^{1}\frac{(-1)^{n}a u^{(2n)} u}{a}\,dx=\int_{0}^{1}u^{(n)}u^{(n)} dx=\int_{0}^{1}(u^{(n)})^2dx\ge 0.
		\end{equation*}
		\indent
		\underline{$I+\tilde{A_n}$ is surjective:} as in Theorem \ref{Theorem 2.1}, one can prove that for $f\in L^2_{\frac{1}{a}}(0,1)$ there exists a unique $u\in H^n_{\frac{1}{a}}(0, 1)$ such that for all $v\in H^n_{\frac{1}{a}}(0, 1)$:
		\begin{equation}\label{Relation (4.13)}
			\left\langle u, v \right\rangle_{H^n_{\frac{1}{a}}(0, 1)}=\int_{0}^{1} \frac{f v}{a}\,dx,
		\end{equation}
		where
		\begin{equation*}
			\left\langle u, v \right\rangle_{H^n_{\frac{1}{a}}(0, 1)} := \int_{0}^{1} \biggl (\frac{u v}{a}+u^{(n)}v^{(n)} \biggr ) dx\quad\,\,\,\,\,\,\,\forall\; u, v \in H^n_{\frac{1}{a}}(0, 1).
		\end{equation*} 
		More precisely, since $\mathcal{C}^{\infty}_c(0, 1) \subset H^n_{\frac{1}{a}}(0, 1)$, the relation (\ref{Relation (4.13)}) holds for all $v\in \mathcal{C}^{\infty}_c(0, 1)$, i.e.,
		\begin{equation*}
			\int_{0}^{1} u^{(n)}v^{(n)}dx = \int_{0}^{1} \frac{f-u}{a}\, v\,dx\quad\,\,\,\,\,\,\,\forall	\;v\in \mathcal{C}^{\infty}_c(0, 1).
		\end{equation*}
		However, the previous equality is equivalent to
		\begin{equation*}
			(-1)^{n}\int_{0}^{1} u^{(2n)}v\,dx = \int_{0}^{1} \frac{f-u}{a}\,v\,dx\Longleftrightarrow  \int_{0}^{1} \frac{(-1)^{n}au^{(2n)}-f+u}{a}\,v\,dx=0,
		\end{equation*}
		for all $v\in \mathcal{C}^{\infty}_c(0, 1)$.
		Evidently,
		\begin{equation*}
			u\in D(A_n)\,\,\,\,\,\,\text{and}\,\,\,\,\,\,u+(-1)^{n}au^{(2n)}=f,
		\end{equation*}
		i.e.,  $(I+\tilde{A_n})(u)=f$. As a consequence, the operator $(\tilde{A_n},D(A_n))$ generates a cosine family and an analytic semigroup of angle $\displaystyle\frac{\pi}{2}$ on $L^2_{\frac{1}{a}}(0, 1)$.
	\end{proof}
	
	As in Section \ref{SECTION 2}, fixed a natural number $n\ge 3$, $T>0$, $u_0\in L^2_{\frac{1}{a}}(0,1)$ and $f\in L^2(0,T;L^2_{\frac{1}{a}}(0,1))$, we have  that the following problem 
	
	\begin{equation*}
		\begin{cases}
			\ds \frac{\partial u}{\partial t}(t,x)+a(x)\frac{\partial^{2n}u}{\partial x^{2n}}(t,x)=f(t,x), &(t,x)\in (0,T) \times (0,1),\\
			u(t,0)=u(t,1)=0,& t\in (0,T),\\
			\ds \frac{\partial^i u}{\partial x^i}(t,0)=\frac{\partial^i u}{\partial x^i}(t,1)=0,&t\in (0,T),  i=1,...,n-1,\\
			u(0,x)=u_0(x),&x\in(0,1),
		\end{cases}
	\end{equation*}
	has a unique solution as a consequence of the previous theorem.

	Finally, fixed $n\in\mathbb{N}$, $n\ge 3$, it is possible to extend the results of Section \ref{Section 3} in a natural way. More precisely, considering the same assumptions on the function $a$, we have the next general proposition which is useful to characterize  $D(A_n)$ in the weakly degenerate setting.
	\begin{Proposition} Fix $n \in \N$, $n\ge 3$, and assume Hypothesis \ref{hp 3.1}. Then, we have the following properties
		\begin{enumerate}
			\item the spaces $H^n_{\frac{1}{a}}(0, 1)$ and $H^n_0(0, 1)$ coincide algebraically; 
			\item the norms $\|\cdot\|_i$ and $\|\cdot\|_{H^i_{0}(0, 1)}$, $i=1,...,n$, are equivalent;
			\item if $u\in H^n_0(0,1)$, i.e., $u \in H^n_{\frac{1}{a}}(0, 1)$, then
			\begin{equation*}
				(au^{(i)})(x_0)=0,\quad\,\,\,i=0, 1,...,n-1.
			\end{equation*}
		\end{enumerate}
	\end{Proposition}
	\begin{proof} One can prove the first two points as in
		Proposition \ref{proposition 3.1}; thus we omit it. We will prove only the last point. To this purpose, we take $u \in H^n_0(0,1)$. Hence $u, u', u'',...u^{(n-1)} \in \mathcal C[0,1]$  and, using the fact that $a(x_0)=0$, one has $(au^{(i)})(x_0)=0$ for all $i =0, 1,...,n-1$. Hence, the thesis follows.
	\end{proof}
	
	We underline that the previous proposition holds if $x_0 \in [0,1]$.
	
	Consequently, the space $\mathcal{C}^{\infty}_c(0, 1)$ is dense in $H^i_{\frac{1}{a}}(0, 1)$, $i\ge 3$, and, in particular, $D(A_n)$ can be rewritten as 
	\[
	\begin{aligned}
		D(A_n)=\biggl \{u\in H^n_0(0,1):\, & (au^{(i)})(x_0)=0,\, i=0, 1,...,n-1,\; au^{(2n)}\in L^2_{\frac{1}{a}}(0, 1) \biggr \},
	\end{aligned}
	\]
	$n \ge 3$.

	The discussion in the strongly degenerate case is based on the introduction of the space
	\begin{equation*}
		X_n:=\biggl \{u\in H^n_{\frac{1}{a}}(0, 1): u(x_0)=0,\,(au^{(i)})(x_0)=0,\,i=1,...,n-1 \biggr \},\,\,\,n\ge 3.
	\end{equation*}
	Observe that $X_2$ coincides with the space $X$ introduced in the previous section.

	Naturally, if $x_0 \in \{0,1\}$, then 
	\[
	X_n= H^n_{\frac{1}{a}}(0, 1).
	\]
	
	The following result is crucial to characterize the domain of the operator in the strongly degenerate case.
	
	\begin{Proposition} Fix $n \in \N$, $n\ge 3$. Assume Hypotheses \ref{hp 3.2} and \ref{hp 3.3} and $x_0\in(0,1)$. Then
		\begin{equation*}
			H^n_{\frac{1}{a}}(0, 1)=X_n
		\end{equation*}
		and the norms $\Vert u \Vert_{n}^2$ and $\int_{0}^{1}(u^{(n)})^2dx$ are equivalent for all $u \in H^n_{\frac{1}{a}}(0, 1)$.
	\end{Proposition}
	\begin{proof}
		Obviously, $X_n \subseteq H^n_{\frac{1}{a}}(0, 1)$. Now, take $u\in H^n_{\frac{1}{a}}(0, 1)$. 
		By Proposition \ref{Prop 3.2}, it holds
		$u(x_0)=(au')(x_0)=0$. Thus, in order to have the thesis, it is sufficient to prove that $(au'')(x_0)=...=(au^{(n-1)})(x_0)=0$.
		To this purpose,
		since $u \in H^n_0(0,1)$, then $u''$,..., $u^{(n-1)}$ and hence $au''$,..., $au^{(n-1)}$ belong to $\mathcal C[0,1]$. Hence, using the fact that $a(x_0)=0$, one has
		\[
		\lim_{x \to x_0} (au^{(i)})(x)=0, \quad i=2,...,n-1.
		\]

		Now, we prove that the two norms are equivalent. To this purpose, take $u \in X_n$. By Lemma \ref{lemma norme equivalenti}, there exists a positive constant $C$ such that
		\begin{equation}\label{dis}
			\Vert u^{(n)} \Vert^2_{L^2(0,1)} \le \Vert u \Vert^2_{n} \le C\bigl ( \Vert u' \Vert^2_{L^2(0,1)}+\Vert u^{(n)} \Vert^2_{L^2(0,1)}\bigr ).
		\end{equation}
		Using the Jensen's inequality and the fact that $u^{(i)}(j)=0$, for all $i=1,..., n-1$ and $j=0,1$, one can prove that
		\[
		\int_0^1 (u^{(i)})^2(x) dx \le \int_0^1 (u^{(i+1)})^2(x) dx, \quad i=1,..., n-1.
		\]
		Thus, by the previous inequality in \eqref{dis}, one can conclude that the two norms are equivalent.
	\end{proof}
	
	Also in this context we underline that, if we know a priori that $u'',...,u^{(n-1)} \in L^2_{\frac{1}{a}}(0,1)$, using the same strategy employed in the previous section, we could prove that $u''(x_0)=...=u^{(n-1)}(x_0)=0$. Hence, the next characterizations hold.
	
	Define 
	\begin{equation*}
		\widetilde H^n_{\frac{1}{a}}(0, 1):= \biggl \{u\in H^n_{\frac{1}{a}}(0, 1): u^{(i)} \in L^2_{\frac{1}{a}}(0,1)\,\,\,\,\,\forall \; i=1,...,n-1 \biggr \}
	\end{equation*}
	and
	\begin{equation*}
		\widetilde X_n:= \biggl \{u\in H^n_{\frac{1}{a}}(0, 1): u^{(i)}(x_0)=0\,\,\,\,\,\forall \; i=0,...,n-1 \biggr \}.
	\end{equation*}

	\begin{Proposition}
		Fix $n \in \N$, $n \ge 3$, and assume Hypotheses \ref{hp 3.2} and \ref{hp 3.3}. Then there exists a positive constant $C$ such that, for all $u \in \widetilde X_n$,
		\begin{equation}\label{disuguaglianze}
			\int_{0}^{1}\frac{(u^{(i)})^2}{a}\,dx \le C \int_{0}^{1}(u^{(i+1)})^2dx
		\end{equation}
		for all $i =2,..., n-1$. Hence
		\begin{equation*}
			\widetilde H^n_{\frac{1}{a}}(0, 1)=\widetilde X_n.
		\end{equation*}
	\end{Proposition}

	\begin{Remark}
		If  Hypotheses \ref{hp 3.2} and \ref{hp 3.3} are satisfied and $u\in H^n_{\frac{1}{a}}(0, 1)$ satisfies the relations (\ref{disuguaglianze}), then $u'',...,u^{(n-1)} \in L^2_{\frac{1}{a}}(0,1)$  and again $u''(x_0)=...=u^{(n-1)}(x_0)=0$ by the previous proposition. Clearly, if $u' \in L^2_{\frac{1}{a}}(0,1)$, then $u'(x_0)=0$  by Remark \ref{rem}.
	\end{Remark}
	
	\begin{Proposition} Fix $n \in \N$, $n\ge 3$, and let
		\begin{equation*}
			\begin{split}
			D_n:=\biggl \{u \in H^n_{\frac{1}{a}}(0, 1):&\, au^{(2n)}\in L^2_{\frac{1}{a}}(0, 1), u(x_0)=0,\\ &(au^{(i)})(x_0)=0\,\,\,\,\,\forall \; i=1,...,n \biggr \}.
				\end{split}
		\end{equation*}
		If Hypotheses \ref{hp 3.2}, \ref{hp 3.3} and \ref{hp 3.3.2} are satisfied, then $D(A_n) =D_n$.
	\end{Proposition}


\begin{thebibliography}{99}
		\bibitem{acf}  Alabau-Boussouira, F., Cannarsa, P., Fragnelli, G.: Carleman estimates for degenerate parabolic operators with application to null controllability. J. Evol. Equ. \textbf{6}, 161-204 (2006).
		\bibitem{akw} Aristotelous, A.C., Karakashian, O.A., Wise, S.M.: Adaptive, second-order in time, primitive-variable discontinuous Galerkin schemes for a Cahn-Hilliard equation with a mass source. IMA J. Numer. Anal. \textbf{35}, 1167--1198 (2015).
		\bibitem{bbdal} Beretta, E., Bertsch, M., Dal Passo, R.: Nonnegative solutions of a fourth-order nonlinear degenerate parabolic equation. Arch. Ration. Mech. Anal. \textbf{129}, 175--200 (1995).
		\bibitem{bf} Bernis, F., Friedman, A.: Higher order nonlinear degenerate parabolic equations. J. Differ. Equ. \textbf{83}, 179--206 (1990).
		\bibitem{be} Bertozzi, A., Esedoglu, S., Gillette, A.: Inpainting of binary images using the Cahn-Hilliard equation. IEEE Trans. Image Process. \textbf{16}, 285--291 (2007).
		\bibitem{bp} Bertozzi, A.L., Pugh, M.: The lubrication approximation for thin viscous films: regularity and long-time behavior of weak
		solutions. Commun. Pure Appl. Math. \textbf{49}, 85--123 (1996).
		\bibitem{cfr} Cannarsa, P., Fragnelli, G., Rocchetti, D.: Null controllability of degenerate parabolic operators with drift. Netw.
		Heterog. Media \textbf{2}, 693--713 (2007).
		\bibitem{2}
		Cannarsa, P., Fragnelli, G., Rocchetti, D.: Controllability results for a class of one-dimensional degenerate
		parabolic problems in nondivergence form. J. Evol. Equ. {\bf 8}, 583--616 (2008).
		\bibitem{c} Chalupeckí, V.: Numerical studies of Cahn-Hilliard equations and applications in image processing, in: Proceedings of Czech-Japanese Seminar in Applied Mathematics, August 4–7, 2004, Czech Technical University in Prague, 2004.
		\bibitem{cm} Cohen, D., Murray, J.M.: A generalized diffusion model for growth and dispersion in a population. J. Math. Biol. \textbf{12}, 237-248 (1981).
		\bibitem{dgg} Dal Passo, R., Garcke, H., Grün, G.: On a fourth-order degenerate parabolic equation: global entropy estimates, existence, and qualitative behavior of solutions. SIAM J. Math. Anal. \textbf{29}, 321--342 (1998).
		\bibitem{dv} Dolcetta, I.C., Vita, S.F.: Area-preserving curve-shortening flows: from phase separation to image processing. Inter-faces Free Bound. \textbf{4}, 325--343 (2002).
		\bibitem{eg} Elliott, C.M., Garcke, H.: On the Cahn–Hilliard equation with degenerate mobility. SIAM J. Math. Anal. \textbf{27},
		404--423 (1996).
		\bibitem{eakds} Erlebacher, J., Aziz, M.J., Karma, A., Dimitrov, N., Sieradzki, K.: Evolution of nanoporosity in dealloying. Nature \textbf{410}, 450--453 (2001).
		\bibitem{3}
		Fragnelli, G., Goldstein, G.R., Goldstein, J.A., Romanelli, S.:  Generators with interior degeneracy on spaces of $L^2$ type. Electron. J. Differ. Equ. {\bf 2012}, 1--30 (2012).
		\bibitem{fgmr} Fragnelli, G., Goldstein, J.A., Mininni, R.M., Romanelli, S.: Operators of order $2n$ with interior degeneracy. Discrete Contin. Dyn. Syst.-S \textbf{13}, 3417-3426 (2020).
		\bibitem{fmAnona2013} Fragnelli, G., Mugnai, D.: Carleman estimates and observability inequalities for parabolic equations with interior degeneracy. Adv. Nonlinear Anal. \textbf{2}, 339-378 (2013).
		\bibitem{4}
		Fragnelli, G., Mugnai, D.: Control of Degenerate and Singular Parabolic Equations. Carleman Estimates and Observability. SpringerBriefs in Mathematics, Springer International Publishing (2021).
		\bibitem{gr} Grün, G.: Degenerate parabolic differential equations of fourth order and a plasticity model with non-local hardening. Z. Anal. Anwend. \textbf{14}, 541--574 (1995).
		\bibitem{gg} Guo, B., Gao, W.: Study of weak solutions for a fourth-order parabolic equation with variable exponent of nonlinearity. Z. Angew. Math. Phys. \textbf{62}, 909--926 (2011).
		\bibitem{ks} Khain, E., Sander, L.M.: A generalized Cahn-Hilliard equation for biological applications. Phys. Rev. E \textbf{ 77}, 051129 (2008).
		\bibitem{ksw} King, B.B., Stein, O., Winkler, M.: A fourth-order parabolic equation modeling epitaxial thin film growth. J. Math. Anal. Appl. \textbf{286}, 459--490 (2003).
		\bibitem{kd} Klapper, I., Dockery, J.: Role of cohesion in the material description of biofilms. Phys. Rev. E \textbf{74}, 0319021 (2006).
		\bibitem{ll} Li, B., Liu, J.G.: Thin film epitaxy with or without slope selection. Eur. J. Appl. Math. \textbf{14}, 713--743 (2003).
		\bibitem{lz} Liang, B., Zheng, S.: Existence and asymptotic behavior of solutions to a nonlinear parabolic equation of fourth order. J. Math. Anal. Appl. \textbf{348}, 234--243 (2008).
		\bibitem{l} Liu, Q.-X., Doelman, A., Rottschäfer, V., de Jager, M., Herman, P.M.J., Rietkerk, M., van de Koppel, J.: Phase separation explains a new class of self-organized spatial patterns in ecological systems. Proc. Natl. Acad. Sci. \textbf{110}, 11905–11910 (2013).
		\bibitem{odb} Oron, A., Davis, S.H., Bankoff, S.G.: Long-scale evolution of thin liquid films. Rev. Mod. Phys. \textbf{69}, 931--980 (1997).
		\bibitem{t} Tremaine, S.: On the origin of irregular structure in Saturn’s rings. Astron. J. \textbf{125}, 894--901 (2003).
		\bibitem{vbd} Verdasca, J., Borckmans, P., Dewel, G.: Chemically frozen phase separation in an adsorbed layer. Phys. Rev. E \textbf{52}, 4616--4619 (1995).
	\end{thebibliography}
\end{document}